\documentclass[11pt,a4paper]{article}
\usepackage{graphicx}
\usepackage{amsmath}
\usepackage{amsfonts}
\usepackage{amssymb}
\usepackage{enumerate}
\usepackage{lscape}
\usepackage{longtable}
\usepackage{url}
\usepackage{color}

\newtheorem{theorem}{Theorem}

\newtheorem{definition}{Definition}

\newtheorem{lemma}{Lemma}

\newtheorem{proposition}{Proposition}

\newcommand{\func}{\operatorname}

\newenvironment{proof}[1][Proof]{\textbf{#1.} }{\ \rule{0.5em}{0.5em} \vspace{1ex}}

\begin{document}

\title{Landau's necessary density conditions for the Hankel transform}
\author{ Lu\'{\i}s Daniel Abreu\thanks{%
Department of Mathematics of University of Coimbra, 3001-454 Coimbra (%
\texttt{daniel@mat.uc.pt}). Currently at NuHAG, University of Vienna, in
(FWF) project \textquotedblleft Frames and Harmonic
Analysis\textquotedblright . This research was partially supported by
CMUC/FCT and FCT project \textquotedblleft Frame Design\textquotedblright\
PTDC/MAT/114394/2009, POCI 2010 and FSE. } \and Afonso S. Bandeira\thanks{%
Program in Applied and Computational Mathematics, Princeton University, NJ
08544, USA (\texttt{ajsb@math.princeton.edu}). Part of this work was done
while the second author was at Department of Mathematics of University of
Coimbra supported by the research grant BII/FCTUC/C2008/CMUC. Partially
supported by CMUC/FCT and FCT project \textquotedblleft Frame
Design\textquotedblright\ PTDC/MAT/114394/2009, POCI 2010 and FSE. } }
\maketitle

\begin{abstract}
{\small We will prove an analogue of Landau's necessary conditions [\emph{%
Necessary density conditions for sampling and interpolation of certain
entire functions}, Acta Math. 117 (1967).] for spaces of functions whose
Hankel transform is supported in a measurable subset $S$ of the positive
semi-axis. As a special case, necessary density conditions for the existence
of Fourier-Bessel frames are obtained.
}
\end{abstract}

\footnotesep=0.4cm

{\small \bigskip }

\begin{center}
\end{center}


\section{Introduction}

While Fourier Series rely on the fact that $\{e^{ikx}\}_{k\in \mathbb{Z}}$
constitutes an orthogonal basis for $L^{2}(-\pi ,\pi )$,\emph{\ Nonharmonic
Fourier Series} allow more general sets $\{e^{it_{k}x}\}_{k\in \mathbb{Z}}$.
They can be nonuniform as in \emph{Riesz Basis }\cite{RMYoung_2001}, perhaps
even redundant as in \emph{Fourier Frames }\cite{JOCerda_KSeip_2002}. On
their \textquotedblleft frequency side\textquotedblright , nonharmonic
Fourier series provide nonuniform and redundant sampling theorems in spaces
of bandlimited functions. As a consequence of Landau's necessary conditions
for sampling and interpolation of such functions \cite%
{HJLandau_1967acta,HJLandau_1967ieee}, we know that sampling requires $%
\{t_{k}\}_{k\in \mathbb{Z}}$ to be \textquotedblleft denser than $\mathbb{Z}$%
\textquotedblright\ and that interpolation requires $\{t_{k}\}_{k\in \mathbb{%
Z}}$ to be \textquotedblleft sparser than $\mathbb{Z}$\textquotedblright .
The set $\mathbb{Z}$ is a sequence of both sampling and interpolation for
bandlimited functions (this is known as the Whittaker-Shannon-Kotel'nikov
sampling theorem).

Likewise, let $J_{\alpha }$ be the Bessel function of order $\alpha >-1/2$
and $j_{n,\alpha }$ its $n^{th}$ zero. Several classical results in Fourier
analysis have been extended to Fourier-Bessel series~\cite%
{OCiaurri_LRoncal_2010,OCiaurri_KStempak_2006,JJGuadalupe_MPerez_FJRuiz_JLVarona_1993,GNWatson_1944}%
. The theory of Fourier-Bessel series is based on the fact that $\{x^{\frac{1%
}{2}}J_{\alpha }(j_{n,\alpha }x)\}_{n=0}^{\infty }$ is an orthogonal basis
for $L^{2}[0,1]$. Thus, the study of more general sets $\{x^{\frac{1}{2}%
}J_{\alpha }(t_nx)\}_{n=0}^{\infty }$ leads naturally to ``nonharmonic
Fourier-Bessel sets''. Completeness properties of such sets have been
investigated by Boas and Pollard \cite{RPBoas_HPollard_1947}, and some
stability results concerning Riesz basis\ have been obtained in \cite%
{MDRawn_1989}. However, the problems that arise naturally in connection with
frame theory and, in particular, questions related to frame and sampling
density, have not been investigated up to the present date. We will address
this question in the present paper, as a special case of a more general
result, which gives Landau-type results in the context of the Hankel
transform.

Throughout this paper $S$ is assumed to be a measurable subset of $(0,\infty
)$.

Consider the space $\mathcal{B}_{\alpha }(S)$ of functions in $%
L^{2}(0,\infty )$ such that their Hankel transform,
\begin{equation*}
H_{\alpha }\left( f\right) (x)=\int_{0}^{\infty }f(t)(xt)^{1/2}J_{\alpha
}(xt)dt\text{,}
\end{equation*}%
is supported in $S$. The special case $S=[0,1]$ is an important example of a
reproducing kernel Hilbert space with an associated sampling theorem~\cite%
{JRHiggins_1972}. Moreover, this reproducing kernel Hilbert space is
strongly reminiscent of the classical Paley-Wiener space of bandlimited
functions. In particular, with a view to solving an eigenvalue problem
arising in the theory of random matrices, Tracy and Widom~\cite%
{CATracy_HWidom_1994bessel} have constructed a set of functions which play
the role of the prolate spheroidal functions in this situation. Such
functions are examples of doubly orthogonal functions in the sense of Stefan
Bergman\ \cite{SBergman_1970}. This automatically implies~\cite{KSeip_1991}
that they solve the concentration problem%
\begin{equation*}
\lambda _{k}\phi _{k}(x)=\int_{0}^{r}\phi _{k}(t)\mathcal{R}_{\alpha
}(t,x)dt,
\end{equation*}%
where $\mathcal{R}_{\alpha }(t,x)$ is the reproducing kernel of $\mathcal{B}%
_{\alpha }([0,1])$. Once we know that such functions exist, it becomes
natural to ask if the behaviour of the corresponding eigenvalues displays
the \textquotedblleft plunging phenomenon\textquotedblright\ which has been
observed in association with the \textquotedblleft Nyquist
rates\textquotedblright\ described in terms of Beurling-type density
theorems (see \cite{HJLandau_1967acta}, \cite{KSeip_1991} and the discussion
in \cite[Chapter~2]{IDaubechies_1992}). We will see that this is indeed the
case, even in the more general case of the space $\mathcal{B}_{\alpha }(S)$,
where $S$ is a measurable subset of the positive semi-axis.

The description of our results requires some terminology. A sequence $%
\Lambda =\{t_{n}\}_{n=0}^{\infty }$ is a \emph{set of sampling} for $%
\mathcal{B}_{\alpha }(S)$ if there exists a constant $A$ such that, for
every $f\in \mathcal{B}_{\alpha }(S)$,
\begin{equation*}
A\int_{0}^{\infty }\left\vert f(x)\right\vert ^{2}dx\leq \sum_{n=0}^{\infty
}|f(t_{n})|^{2}\text{.}
\end{equation*}%
Moreover, $\Lambda $ is a \emph{set of interpolation} for $\mathcal{B}%
_{\alpha }(S)$ if, given any set of numbers $\{a_{n}\}_{n=0}^{\infty }$ with
$\sum \left\vert a_{n}\right\vert ^{2}<\infty $, there exists $f\in \mathcal{%
B}_{\alpha }(S)$ such that
\begin{equation*}
f(t_{n})=a_{n}\text{, for every }t_{n}\in \Lambda \text{.}
\end{equation*}%
We say that a sequence is separated if the distance between any two distinct
points exceeds some positive quantity $d>0$. For such sequences we can
define densities which are suitable for analysis of functions supported in $%
(0,\infty ).$

\begin{definition}
\label{definitiondensity} Let $n_a(r)$ denote the number of points of $%
\Lambda \subset $ $(0,\infty )$ to be found in $\left[ a,a+r\right] $. Then
the \emph{lower }and the \emph{upper densities} of $\Lambda $ are given by
the limits%
\begin{equation*}
D^{-}(\Lambda )=\lim_{r\rightarrow \infty }\inf \inf_{a\geq 0}\frac{n_{a}(r)%
}{r}\text{ \ \ and \ \ \ \ \ }D^{+}(\Lambda )=\lim_{r\rightarrow \infty
}\sup \sup_{a\geq 0}\frac{n_{a}(r)}{r}\text{.}
\end{equation*}
\end{definition}


Our main results read as follows.

\begin{theorem}
Let $S$ be a measurable subset of $(0,\infty)$ and $\alpha >-1/2$. If a
separated set $\Lambda $ is of sampling for $\mathcal{B}_{\alpha }(S)$, then
\begin{equation}
D^{-}(\Lambda )\geq \frac{1}{\pi }m(S)\text{.}  \label{estdensity}
\end{equation}
\end{theorem}

\begin{theorem}
Let $S$ be a bounded measurable subset of $(0,\infty)$ and $\alpha >-1/2$.
If the set $\Lambda $ is of interpolation for $\mathcal{B}_{\alpha }(S)$,
then
\begin{equation}
D^{+}(\Lambda )\leq \frac{1}{\pi }m(S)\text{.}
\label{estdensityinterpolation}
\end{equation}
\end{theorem}

A major technical difficulty in the proofs of the above results arises from
the translation invariance of Definition 1, since we cannot appeal to the
translation invariance of the eigenvalue problem which was used by Landau in
\cite{HJLandau_1967acta}. For this reason, delicated estimates of operators
involving the reproducing kernels of the space $B_{\alpha }([a,a+r])$ are
required.

We will also prove that the separation condition imples the existence of a
constant $B$ such that, for every $f\in B_{\alpha }(S)$,%
\begin{equation*}
\sum_{n}\left\vert g(t_{n})\right\vert ^{2}\leq B\Vert g\Vert ^{2}.
\end{equation*}%
Theorem 1 can be seen from the \emph{frame theory} viewpoint\emph{.} A
sequence of functions $\{e_{j}\}_{j\in I}$ is said to be a \emph{frame} in a
Hilbert space $H$ if there exist positive constants $A$ and $B$ such that,
for every $f\in H$,
\begin{equation}
A\left\Vert f\right\Vert _{H}^{2}\leq \sum_{j\in I}\left\vert \left\langle
f,e_{j}\right\rangle \right\vert ^{2}\leq B\left\Vert f\right\Vert _{H}^{2}.
\label{frame}
\end{equation}%
Accordingly, we say that $\{(t_{n}x)^{\frac{1}{2}}J_{\alpha }(t_{n}x)\}$ is
a \emph{Fourier-Bessel frame} if there exist positive constants $A$ and $B$
such that, for every $f\in \mathcal{B}_{\alpha }[(0,1)]$,%
\begin{equation*}
A\int_{0}^{1}\left\vert f(x)\right\vert ^{2}dx\leq \sum_{n=0}^{\infty
}\left\vert \int_{0}^{1}(t_{n}x)^{\frac{1}{2}}f(x)J_{\alpha
}(t_{n}x)dx\right\vert ^{2}\leq B\int_{0}^{1}\left\vert f(x)\right\vert
^{2}dx\text{.}
\end{equation*}%
By choosing $S=(0,1)$ in Theorem 1 one concludes that, if $\{(t_{n}x)^{\frac{%
1}{2}}J_{\alpha }(t_{n}x)\}$ is a Fourier-Bessel frame, then\ $D^{-}(\Lambda
)\geq \frac{1}{\pi }$. In particular, the orthogonal basis $\{(j_{n,\alpha
}x)^{\frac{1}{2}}J_{\alpha }(j_{n,\alpha }x)\}$ is a Fourier-Bessel frame,
since the norm of each element is bounded away from zero and infinity~\cite[%
(2.3)]{JJGuadalupe_MPerez_FJRuiz_JLVarona_1993}. It is well known (as a
consequence of the Paley-Wiener theorem: see \cite[Theorem 2]{JRHiggins_1972}%
) that if $f\in \mathcal{B}_{\alpha }[(0,1)]$ then $t^{-\alpha -\frac{1}{2}%
}f(t)$ belongs to the Paley-Wiener space. Thus, every sufficient condition
for Fourier frames also holds in the case of Fourier-Bessel frames. For an
account of such conditions see , for instance, those in \emph{\ }\cite[pg.
791]{JOCerda_KSeip_2002} and the references therein. Such an observation may
be useful in the construction of the \textquotedblleft Bessel
analogues\textquotedblright\ of the hyperbolic lattice in \cite[Theorem 3.4]%
{HRauhut_MRosler_2005}

Recently, Marzo \cite{JMarzo2007} applied Landau's ideas to the proof of
Marcinkiewicz--Zygmund inequalities in the sphere. From his work we borrow
an idea to start with the estimations leading to inequality (\ref{estdensity}%
) and a method to deal with the case where a sequence of both sampling and
interpolation is unknown (as in our more general situation) or do not exist
(the case for higher dimensions in \cite{JMarzo2007}). It is worth noting
that analogues of Landau's necessary conditions have been also studied \cite%
{KGrochenig_HRazafinjatovo_1996,KGrochenig_GKutyniok_KSeip_2008} using
techniques from time-frequency analysis \cite{JRamanathan_TSteger_1995}.

The outline of the paper is as follows. In Section 2 we collect some results
about the convolution structure associated with the Hankel transform. The
key section is Section 3, where the eigenvalue problem is formulated and the
estimates of the trace and norm are obtained. Section 4 contains the lemmas
which are required, in the proofs of the main results, to establish the
connection between the sampling and interpolation concepts and the
eigenvalue problem. We prove our main results in Section 5.

\section{Bessel functions and their convolution structure}

In this section we will use \cite{HRauhut_2004} and \cite%
{HRauhut_MRosler_2005} as reference sources for some definitions and
properties that are useful in the harmonic analysis associated with the
Hankel transform. For $\alpha >-1$, the \emph{Bessel functions} are defined
by the power series,
\begin{equation*}
J_{\alpha }(x)=\sum_{n=0}^{\infty }\frac{(-1)^{n}(\frac{x}{2})^{2n+\alpha }}{%
n!\Gamma (n+\alpha +1)}.
\end{equation*}%
Bessel functions are solutions of the second order differential equation
\begin{equation}
\frac{d^{2}y}{dx^{2}}+\frac{1}{x}\frac{dy}{dx}+\left( 1-\frac{\alpha ^{2}}{%
x^{2}}\right) y=0.  \label{eqdifbessel}
\end{equation}%
The derivative of a Bessel function can be related to a Bessel function of
different order via the formula
\begin{equation}
\frac{1}{x^{m}}\left( \frac{d}{dx}\right) ^{m}J_{\alpha }(x)=x^{\alpha
-m}J_{\alpha -m}(x),  \label{formderbessel}
\end{equation}%
valid for every positive integer $m$. We will make extensive use of the
asymptotic formulae \cite{GNWatson_1944}:
\begin{equation}
J_{\alpha }(x)=\sqrt{\frac{2}{\pi x}}\left( \sin \eta _{x}+\rho (x)\right) ,
\label{assimpt}
\end{equation}%
with $\rho (x)=\mathcal{O}(x^{-1})$, where $\eta _{x}=x-\left( \frac{1}{2}%
\alpha -\frac{1}{4}\right) \pi $, and
\begin{equation}
J_{\alpha }^{\prime }(x)=\sqrt{\frac{2}{\pi x}}\left( \cos \eta _{x}+\rho
_{1}(x)\right)  \label{assimptder}
\end{equation}%
with $\rho _{1}(x)=\mathcal{O}(x^{-1})$. 
Sometimes it is convenient to renormalize the Bessel functions in the
following way:{\normalsize \
\begin{equation*}
j_{\alpha }(x)=\Gamma (\alpha +1)\left( \frac{2}{x}\right) ^{\alpha
}J_{\alpha }(x)\text{.}
\end{equation*}%
} The functions $j_{\alpha }$\ are the \emph{spherical} Bessel functions.
They satisfy{\normalsize \ $j_{\alpha }(0)=1$ }and $|j_{\alpha }(x)|\leq 1$,
for all $x\in (0,\infty )$. For the Harmonic Analysis associated with the
Hankel transform one defines a \textquotedblleft Hankel
modulation\textquotedblright\ $\left( \func{m}_{\lambda }f\right)
(x)=j_{\alpha }(\lambda x)f(x)$ and associates with it a \textquotedblleft
Hankel translation\textquotedblright\ $H_{\alpha }\left( \tau _{\lambda
}f\right) (x)=\left( \func{m}_{\lambda }H_{\alpha }f\right) (x)=j_{\alpha
}(\lambda x)(H_{\alpha }f)(x)$. This allows to define a \emph{%
\textquotedblleft Hankel convolution\textquotedblright } as follows:

{\normalsize
\begin{equation*}
f\ast _{\alpha }g(\lambda )=\lambda ^{\alpha +\frac{1}{2}}\left( \frac{\pi }{%
2}\right) ^{\frac{\alpha }{2}}\frac{1}{\Gamma (\alpha +1)}\int_{0}^{\infty
}f(t)\tau _{\lambda }g(t)dt\text{.}
\end{equation*}%
}Hankel convolutions are mapped in products via the formula{\normalsize
\begin{equation*}
H_{\alpha }(f\ast _{\alpha }g)(x)=x^{-\left( \alpha +\frac{1}{2}\right)
}(2\pi )^{\frac{\alpha }{2}}H_{\alpha }f(x)H_{\alpha }g(x)\text{.}
\end{equation*}%
}The following property of Hankel translations, which can be found, for
instance, in \cite{HRauhut_2004}, will be also required: if $\func{supp}%
g\subset \lbrack 0,d]$ and $r>d$ then
\begin{equation}
\func{supp}\tau _{r}g\subset \lbrack \max \{0,r-d\},r+d].
\label{Proposition1}
\end{equation}

\section{The eigenvalue problem}

Let $S$ be a finite union of intervals and $I$ be the interval $I=[a,a+r]$.

Let $D(I)$ be the subspace of $L^{2}(0,\infty )$ consisting of functions
supported on $I$ and $\chi _{I}$ the characteristic function of $I$. Let $%
D_{I}$ and $B_{S}$ denote the orthogonal projections of $L^{2}(0,\infty )$
onto $D(I)$ and $\mathcal{B}_{\alpha }(S)$, respectively. They are given
explicitly by
\begin{equation*}
D_{I}f=\chi _{I}f\text{ \ \ \ \ \ \ \ \ and\ \ \ \ \ \ \ }B_{S}f=H_{\alpha
}D_{S}H_{\alpha }f
\end{equation*}%
We want to maximize, over the functions $f\in \mathcal{B}_{\alpha }(S)$, the
\textquotedblleft energy concentration\textquotedblright\ $\lambda _{f}$
given as
\begin{equation*}
\lambda _{f}=\frac{\int_{I}|f(t)|^{2}dt}{\Vert f\Vert ^{2}}.
\end{equation*}%
This is a standard problem of maximizing a quadratic form and leads to the
eigenvalue problem%
\begin{equation}
\lambda _{k}(I,S)\phi _{k}(x)=B_{S}D_{I}\phi _{k}\text{.}
\label{igualdadevaloresproprios1}
\end{equation}%
Writing the operators explicitly and interchanging the integrals, (\ref%
{igualdadevaloresproprios1}) becomes
\begin{equation}
\lambda _{k}(I,S)\phi _{k}(x)=\int_{I}\phi _{k}(t)w_{S}(t,x)dt,
\label{eigenvaluetempo}
\end{equation}%
where, for a set $X$, the Reproducing Kernel $w_{X}(t,x)$ is given by%
\begin{equation}
w_{X}(t,x)=\int_{X}J_{\alpha }(ts)J_{\alpha }(xs)(tx)^{\frac{1}{2}}sds.
\label{defW}
\end{equation}%
Multiplying both sides of (\ref{eigenvaluetempo}) by $(xu)^{\frac{1}{2}%
}J_{\alpha }(xu)$, integrating with respect to $dx$ in $I$ and changing the
order of the integrals, gives the dual problem of concentrating on $S$
functions whose Hankel Transform is supported on $I$:
\begin{equation}
\lambda _{k}(I,S)\psi _{k}(t)=\int_{S}\psi _{k}(x)w_{I}(x,t)dx\text{.}
\label{eigenvaluefreq}
\end{equation}%
Using this duality and a change of variables gives, for $\beta >0$, the
identities {\normalsize
\begin{eqnarray}
\lambda _{k}(I,S) &=&\lambda _{k}(S,I)  \label{inverterpapeis} \\
&=&\lambda _{k}(\beta I,\beta ^{-1}S)\text{.}  \label{rescalar}
\end{eqnarray}%
}Now set
\begin{equation}
\mathcal{R}_{\alpha }(t,x)=w_{[0,1]}(t,x)=\left\{
\begin{array}{lcl}
(tx)^{\frac{1}{2}}\frac{J_{\alpha }(t)xJ_{\alpha }^{\prime }(x)-J_{\alpha
}(x)tJ_{\alpha }^{\prime }(t)}{t^{2}-x^{2}} & \text{if} & t\neq x \\
\frac{1}{2}\left( xJ_{\alpha }^{\prime }(x)^{2}-xJ_{\alpha }(x)J_{\alpha
}^{\prime \prime }(x)-J_{\alpha }(x)J_{\alpha }^{\prime }(x)\right) & \text{%
if} & t=x.%
\end{array}%
\right.  \label{particularrk}
\end{equation}%
and observe that

\begin{equation}
w_{[a,a+r]}(t,x)=(a+r)\mathcal{R}_{\alpha }((a+r)t,(a+r)x)-a\mathcal{R}%
_{\alpha }(at,ax).
\end{equation}%
We will first study the case when $S$ is a finite union of intervals.
Suppose $S$ to consist of $n$ disjoint intervals $%
(b_{1},b_{1}+s_{1}),...,(b_{n},b_{n}+s_{n})$ and write $s=s_{1}+...+s_{n}$.
As in \cite{HJLandau_1967acta}, the cornerstone of the proofs consists of
Norm and Trace estimates of the above operators. From the above
considerations one has
\begin{eqnarray}
\func{Trace} &=&\sum \lambda _{k}(I,S)=\int_{S}w_{[a,a+r]}(x,x)dx  \notag \\
&=&\int_{S}(a+r)\mathcal{R}_{\alpha }((a+r)x,(a+r)x)-a\mathcal{R}_{\alpha
}(ax,ax)dx \\
&=&\sum_{i=1}^{n}\int_{b_{i}}^{b_{i}+s_{i}}(a+r)\mathcal{R}_{\alpha
}((a+r)x,(a+r)x)-a\mathcal{R}_{\alpha }(ax,ax)dx,
\end{eqnarray}%
and
\begin{eqnarray}
\func{Norm} &=&\sum \lambda
_{k}^{2}(I,S)=\int_{S}\int_{S}w_{[a,a+r]}^{2}(t,x)dtdx  \notag \\
&=&\int_{S}\int_{S}\left[ (a+r)\mathcal{R}_{\alpha }((a+r)t,(a+r)x)-a%
\mathcal{R}_{\alpha }(at,ax)\right] ^{2}dtdx \\
&=&\sum_{i=1}^{n}\sum_{j=1}^{n}\int_{b_{i}}^{b_{i}+s_{i}}%
\int_{b_{j}}^{b_{j}+s_{j}}\left[ (a+r)\mathcal{R}_{\alpha }((a+r)t,(a+r)x)-a%
\mathcal{R}_{\alpha }(at,ax)\right] ^{2}dtdx.
\end{eqnarray}

\subsection{Estimation of $\func{Trace}$}

In this Section we will obtain the following estimation.

\begin{lemma}
\label{lemma_traceestimation} For $\alpha >-1/2$, $\func{Trace}=\frac{1}{\pi
}rs+\mathcal{O}(1)$. More precisely, there exists a constant $L$ such that
for all positive $a,r,b,s$ we have
\begin{equation}
\left\vert \func{Trace}-\frac{1}{\pi }rs\right\vert \leq L.
\end{equation}
\end{lemma}

\begin{proof}
We will first estimate the function%
\begin{equation}
\func{T}(u)=\int_{0}^{u}\mathcal{R}_{\alpha }(x,x)dx=\frac{1}{2}%
\int_{0}^{u}\left( xJ_{\alpha }^{\prime }(x)^{2}-xJ_{\alpha }(x)J_{\alpha
}^{\prime \prime }(x)-J_{\alpha }(x)J_{\alpha }^{\prime }(x)\right) dx
\end{equation}%
For small $x$, the power series expansion of the Bessel function gives
\begin{equation*}
J_{\alpha }(x)=\frac{x^{\alpha }}{2^{\alpha }\Gamma ({\alpha +1)}}+\mathcal{O%
}(x^{\alpha +2}),
\end{equation*}%
leading, for small $u$, to the estimate
\begin{equation*}
\func{T}(u)=\frac{1}{2}\int_{0}^{u}\mathcal{O}(x^{2\alpha +1})dx.
\end{equation*}%
Therefore, for $\alpha >-1/2$ and small $u$, the integral defining $\func{T}%
(u)$ is convergent. We proceed to estimate $\func{T}(u)$.\ Using (\ref%
{assimpt}), (\ref{assimptder}) and (\ref{particularrk}) one obtains, after
some simplification,
\begin{equation}
\mathcal{R}_{\alpha }(x,x)=\frac{1}{\pi }+\epsilon (x),
\end{equation}%
with $|\epsilon (x)|=\mathcal{O}\left( \frac{1}{x}\right) $. It is possible
(and it will be required for our purposes) to improve this estimate even
more, taking into account the cancelations resulting from the changes in
sign of $\epsilon (x)$. Use the second order differential equation (\ref%
{eqdifbessel}) to rewrite $\mathcal{R}_{\alpha }(x,x)$ as
\begin{eqnarray*}
\mathcal{R}_{\alpha }(x,x) &=&\frac{1}{2}\left( xJ_{\alpha }^{\prime
}(x)^{2}-xJ_{\alpha }(x)\left( -\frac{1}{x}J_{\alpha }^{\prime
}(x)-J_{\alpha }(x)+\frac{\alpha ^{2}}{x^{2}}J_{\alpha }(x)\right)
-J_{\alpha }(x)J_{\alpha }^{\prime }(x)\right) \\
&=&\frac{1}{2}\left( xJ_{\alpha }(x)^{2}+xJ_{\alpha }^{\prime }(x)^{2}-\frac{%
\alpha ^{2}}{x}J_{\alpha }(x)^{2}\right) \\
&=&\frac{1}{2}\left( xJ_{\alpha }(x)^{2}+xJ_{\alpha }^{\prime
}(x)^{2}\right) +\mathcal{O}\left( \frac{1}{x^{2}}\right) \\
&=&\frac{1}{2}\left( xJ_{\alpha }(x)^{2}+x\left( J_{\alpha -1}(x)-\frac{%
\alpha }{x}J_{\alpha }(x)\right) ^{2}\right) +\mathcal{O}\left( \frac{1}{%
x^{2}}\right) \\
&=&\frac{1}{2}xJ_{\alpha }(x)^{2}+\frac{1}{2}xJ_{\alpha -1}(x)^{2}-\alpha
J_{\alpha -1}(x)^{2}J_{\alpha }(x)+\mathcal{O}\left( \frac{1}{x^{2}}\right) .
\end{eqnarray*}%
The third and fifth equalities in the above calculation were obtained using (%
\ref{assimpt}) and the fourth one using (\ref{formderbessel}). Now observe
that $\int_{0}^{u}xJ_{\alpha }(x)^{2}dx=\mathcal{R}_{\alpha }(u,u)=\frac{1}{%
\pi }+\mathcal{O}\left( \frac{1}{u}\right) $, and that the same is true if
one replaces $\alpha $ by $\alpha -1$. Moreover, from (\ref{assimpt}) we
obtain $\int_{0}^{u}J_{\alpha -1}(x)^{2}J_{\alpha }(x)=\mathcal{O}(1)$. We
conclude that
\begin{equation}
\func{T}(u)=\frac{1}{\pi }u+\mathcal{O}(1).  \label{T(u)}
\end{equation}

Finally,
\begin{eqnarray*}
\func{Trace} &=&\sum_{i=1}^{n}\int_{b_{i}}^{b_{i}+s_{i}}(a+r)\mathcal{R}%
_{\alpha }((a+r)x,(a+r)x)-a\mathcal{R}_{\alpha }(ax,ax)dx \\
&=&\sum_{i=1}^{n}\left( \int_{b_{i}}^{b_{i}+s_{i}}(a+r)\mathcal{R}_{\alpha
}((a+r)x,(a+r)x)dx-\int_{b_{i}}^{b_{i}+s_{i}}a\mathcal{R}_{\alpha
}(ax,ax)dx\right) \\
&=&\sum_{i=1}^{n}\left( \int_{(a+r)b_{i}}^{(a+r)(b_{i}+s_{i})}\mathcal{R}%
_{\alpha }(x,x)dx-\int_{ab_{i}}^{a(b_{i}+s_{i})}\mathcal{R}_{\alpha
}(x,x)dx\right) \\
&=&\sum_{i=1}^{n}\left( (a+r)s_{i}-as_{i}+\mathcal{O}(1)\right) =\frac{1}{%
\pi }rs+\mathcal{O}(1),
\end{eqnarray*}%
entering the estimate (\ref{T(u)}) in the fourth identity. This is the
required result.
\end{proof}

\subsection{Estimation of $\func{Norm}$}

This section contains the key step, which is the estimation of
\begin{equation}
\func{Norm}=\sum_{i,j}^{n}\int_{b_{i}}^{b_{i}+s_{i}}%
\int_{b_{j}}^{b_{j}+s_{j}}\left[ (a+r)\mathcal{R}_{\alpha }((a+r)t,(a+r)x)-a%
\mathcal{R}_{\alpha }(at,ax)\right] ^{2}dtdx.  \label{Norm_sum}
\end{equation}

\begin{proposition}
\label{norm_complete_estimation} If $\alpha >-1/2$, the function $\func{Norm}
$ satisfies the estimate%
\begin{equation}
\func{Norm}\geq \frac{1}{\pi }rs-K\log (r)-L\text{.}  \label{estN}
\end{equation}%
for some constants $K$ and $L$ not depending on $r$ or $a$.
\end{proposition}

\subsubsection{Preparation Lemmas}

We divide the technical parts of the proof of Theorem \ref%
{norm_complete_estimation} in a few preparation Lemmas.

\begin{lemma}
\label{lemma_normestimation_T} There exists a constant $L$ such that, for
any positive $b,s,a,r$, we have the following
\begin{equation*}
\int_{a}^{a+r}\int_{0}^{\infty }\left[ (b+s)\mathcal{R}_{\alpha
}((b+s)t,(b+s)x)-b\mathcal{R}_{\alpha }(bt,bx)\right] ^{2}dtdx\geq \frac{1}{%
\pi }rs-L.
\end{equation*}%
%
%
%
%
%
\end{lemma}

\begin{proof}
The result will be derived by proving that
\begin{equation}
\int_{a}^{a+r}\int_{0}^{\infty }\left[ (b+s)\mathcal{R}_{\alpha
}((b+s)t,(b+s)x)-b\mathcal{R}_{\alpha }(bt,bx)\right] ^{2}dtdx
\label{expression_trace_appendix_0}
\end{equation}%
is equal to
\begin{equation}
\int_{a}^{a+r}(b+s)\mathcal{R}_{\alpha }((b+s)x,(b+s)x)-a\mathcal{R}_{\alpha
}(bx,bx)dx.  \label{expression_trace_appendix}
\end{equation}%
and then using Lemma \ref{lemma_traceestimation} to bound (\ref%
{expression_trace_appendix}). In order to prove that (\ref%
{expression_trace_appendix_0}) is equal to (\ref{expression_trace_appendix})
we first notice that, if $a\leq b$ then%
\begin{equation*}
\int_{0}^{\infty }ab\mathcal{R}_{\alpha }(ax,at)\mathcal{R}_{\alpha
}(bx,bt)dt=b\int_{0}^{\infty }\frac{a}{b}\mathcal{R}_{\alpha }\left( \frac{a%
}{b}z,\frac{a}{b}t\right) \mathcal{R}_{\alpha }(z,t)dt,
\end{equation*}%
where $z=bx$. Since $\frac{a}{b}\mathcal{R}_{\alpha }\left( \frac{a}{b}z,%
\frac{a}{b}t\right) $ is the reproducing kernel of $\mathcal{B}_{\alpha }(0,%
\frac{a}{b})$, $t\rightarrow \frac{a}{b}\mathcal{R}_{\alpha }\left( \frac{a}{%
b}z,\frac{a}{b}t\right) $ is a function in $\mathcal{B}_{\alpha }(0,\frac{a}{%
b})$ and thus in $\mathcal{B}_{\alpha }(0,1)$, since $\frac{a}{b}\leq 1$.
Using the reproducing kernel property in $\mathcal{B}_{\alpha }(0,1)$ one
gets

\begin{equation}
\int_{0}^{\infty }\frac{a}{b}\mathcal{R}_{\alpha }\left( \frac{a}{b}z,\frac{a%
}{b}t\right) \mathcal{R}_{\alpha }(z,t)dt=\frac{a}{b}\mathcal{R}_{\alpha
}\left( \frac{a}{b}z,\frac{a}{b}z\right) .
\end{equation}%
We just proved that, if $a\leq b$,
\begin{equation}
\int_{0}^{\infty }ab\mathcal{R}_{\alpha }(ax,at)\mathcal{R}_{\alpha
}(bx,bt)dt=a\mathcal{R}_{\alpha }\left( ax,ax\right)
\label{reproducingkernelpropertyalessb}
\end{equation}%
Then, as $b\leq b+s$, expanding (\ref{expression_trace_appendix_0}) and
using (\ref{reproducingkernelpropertyalessb}) on each of the 3 terms, we get
the desired equality.
\end{proof}

\begin{lemma}
\label{lemma_normestimation_PQ} Let
\begin{equation*}
P(a,r)=\int_{0}^{a}\int_{a}^{a+r}\mathcal{R}_{\alpha }^{2}(t,x)dtdx.
\end{equation*}%
and
\begin{equation*}
Q(a,r)=\int_{a}^{a+r}\int_{a+r}^{\infty }\mathcal{R}_{\alpha }^{2}(t,x)dtdx.
\end{equation*}

Then, there exists constants $K,L,K^{\prime }$ and $L^{\prime }$ such that,
for every $a$ and $r$,
\begin{equation}
P(a,r)\leq K\log r+L,
\end{equation}%
and
\begin{equation}
Q(a,r)\leq K^{\prime }\log r+L^{\prime }.
\end{equation}
\end{lemma}

\begin{proof}
Consider the following integrals%
\begin{eqnarray*}
H_{1}(y,u) &=&\int_{y-u}^{y}\int_{y+u}^{\infty }\mathcal{R}_{\alpha
}^{2}(t,x)dtdx \\
H_{2}(y,u) &=&\int_{0}^{y-u}\int_{y}^{y+u}\mathcal{R}_{\alpha }^{2}(t,x)dtdx.
\\
I(y,u) &=&\int_{y-u}^{y}\int_{y}^{y+u}\mathcal{R}_{\alpha }^{2}(t,x)dtdx.
\end{eqnarray*}%
We will prove that $H_{1}(y,u)$ and $H_{2}(y,u)$ are both uniformly bounded
for $y\geq u$ and that there exist constants $K$ and $L$ such that, for
every $y$ and $u$ with $y\geq u$,
\begin{equation}
I(y,u)\leq K\log u+L.  \label{bound_I}
\end{equation}%
The bound on $Q$ is then obtained by noticing that
\begin{equation*}
Q(a,r)=H_{1}(a+r,r)+H_{2}(a+r,r).
\end{equation*}%
To estimate $P$ we have to separate in cases: if $a<r$ then%
\begin{equation*}
P(a,r)\leq I(r,r)\leq K\log r+L,
\end{equation*}%
and if $a\geq r$ we have
\begin{equation*}
P(a,r)=H_{2}(a,r)+I(a,r)\leq K\log r+L.
\end{equation*}%
We will organize the estimates in two steps: the first one contains the
estimates of $H_{1}(y,u)$ and $H_{2}(y,u)$ and the second one of $I(y,u).$

\textbf{Step 1.} Using formulas (\ref{assimpt}) and (\ref{assimptder}) one
can assure the existence of constants $K$ and $K^{\prime }$ such that
\begin{equation}
\mathcal{R}_{\alpha }^{2}(t,x)\leq K\frac{1}{(x-t)^{2}}+K^{\prime }\frac{1}{%
(x^{2}-t^{2})^{2}},  \label{Decomp_sum}
\end{equation}%
for all non-negative $t,x$. For $y\leq 1$ the result easily follows from the
estimate of the $\func{Trace}$. Let us consider $y>1.$ Using (\ref%
{Decomp_sum}) in the definition of $H_{1}(y,u)$\ we obtain contants $K$ and $%
K^{\prime }$ such that
\begin{eqnarray*}
H_{1}(y,u) &\leq &K\int_{y-u}^{y}\int_{y+u}^{\infty }\left( \frac{1}{x-t}%
\right) ^{2}dxdt+K^{\prime }\int_{y-u}^{y}\int_{y+u}^{\infty }\left( \frac{1%
}{x^{2}-t^{2}}\right) ^{2}dxdt \\
&\leq &Ku\int_{y+u}^{\infty }\left( \frac{1}{x-y}\right) ^{2}dx+K^{\prime
}u\int_{y+u}^{\infty }\left( \frac{1}{x-y}\right) ^{2}\left( \frac{1}{x+y}%
\right) ^{2}dx \\
&=&Ku\int_{u}^{\infty }\left( \frac{1}{\theta }\right) ^{2}d\theta
+K^{\prime }u\int_{u}^{\infty }\left( \frac{1}{\theta }\right) ^{2}\left(
\frac{1}{\theta +2y}\right) ^{2}d\theta \\
&\leq &(K+K^{\prime })u\int_{u}^{\infty }\left( \frac{1}{\theta }\right)
^{2}d\theta \\
&=&(K+K^{\prime })\int_{1}^{\infty }\left( \frac{1}{\tau }\right) ^{2}d\tau .
\end{eqnarray*}%
The estimate of $H_{2}(y,u)$ follows the same lines.

\textbf{Step 2. }A change of variables in the double integral results in
\begin{equation*}
I(y,u)=\int_{1-\frac{u}{y}}^{1}\int_{1}^{1+\frac{u}{y}}y^{2}\mathcal{R}%
_{\alpha }^{2}(yt,yx)dtdx.
\end{equation*}%
Writing the integral explicitly and inserting asymptotic formulas (\ref%
{assimpt}) and (\ref{assimptder}), one sees that
\begin{equation*}
I(y,u)=\int_{1-\frac{u}{y}}^{1}\int_{1}^{1+\frac{u}{y}}\left( \mathcal{L}%
^{y}(t,x)+\mathcal{E}^{y}(t,x)\right) ^{2}dtdx\text{,}
\end{equation*}%
with
\begin{equation}
\mathcal{L}^{y}(t,x)=\frac{2}{\pi }\frac{\sin \eta _{yx}t\cos \eta
_{yt}-\sin \eta _{yt}x\cos \eta _{yx}}{x^{2}-t^{2}}
\label{definicaodomathcalL}
\end{equation}%
and $\mathcal{E}^{y}(t,x)=\mathcal{O}(y^{-1})$. As a result (and keeping in
mind that $u\leq y$), $I(y,u)\leq 2\tilde{I}(y,u)+L_{0},$for some $L_{0}$
indepedent of $y$ and $u$, where
\begin{equation*}
\tilde{I}(y,u)=\int_{1-\frac{u}{y}}^{1}\int_{1}^{1+\frac{u}{y}}\left(
\mathcal{L}^{y}(t,x)\right) ^{2}dtdx.
\end{equation*}%
Writing $k_{\alpha }=-\left( \frac{1}{2}\alpha -\frac{1}{4}\right) \pi $ and
using (\ref{definicaodomathcalL}), gives
\begin{eqnarray*}
\tilde{I}(y,u) &=&\int_{1-\frac{u}{y}}^{1}\int_{1}^{1+\frac{u}{y}}\left(
\frac{\sin (yt+k_{\alpha })x\cos (yx+k_{\alpha })-\sin (yx+k_{\alpha })t\cos
(yt+k_{\alpha })}{t^{2}-x^{2}}\right) ^{2}dtdx \\
&=&\int_{1-\frac{u}{y}}^{1}\int_{1}^{1+\frac{u}{y}}\left( \frac{x\sin
(y(t-x))+(x-t)\sin (yx+k_{\alpha })\cos (yt+k_{\alpha })}{t^{2}-x^{2}}%
\right) ^{2}dtdx \\
&=&\int_{1-\frac{u}{y}}^{1}\int_{1}^{1+\frac{u}{y}}\left( \frac{1}{t+x}%
\right) ^{2}\left( y\func{sinc}\left( \frac{y}{\pi }(t-x)\right) -\sin
(yx+k_{\alpha })\cos (rt+k_{\alpha })\right) ^{2}dtdx\text{,}
\end{eqnarray*}%
where we are using the usual notation $\func{sinc}(x)=\sin (\pi x)/\pi x$.
From the last expression it follows that there exists positive constants $A$
and $B$, independent of $y$ and $u$, such that
\begin{equation*}
\tilde{I}(y,u)\leq A\int_{1-\frac{u}{y}}^{1}\int_{1}^{1+\frac{u}{y}}\left( y%
\func{sinc}\left( \frac{y}{\pi }(t-x)\right) \right) ^{2}dtdx+B\leq A\pi ^{2}%
\mathcal{S}(y,u)+B,
\end{equation*}%
where the second inequality is obtained doing a change of variables and
writing
\begin{equation*}
\mathcal{S}(y,u)=\int_{\frac{y-u}{\pi }}^{\frac{y}{\pi }}\int_{\frac{y}{\pi }%
}^{\frac{y+u}{\pi }}\left( \func{sinc}(t-x)\right) ^{2}dtdx.
\end{equation*}%
Yet another change of variables shows that
\begin{eqnarray*}
\mathcal{S}(y,u) &=&\mathcal{S}(u,u)\leq \int_{0}^{\frac{u}{\pi }}\int_{%
\frac{u}{\pi }}^{\infty }\left( \func{sinc}(t-x)\right) ^{2}dtdx+\int_{0}^{%
\frac{u}{\pi }}\int_{\mathbb{R}^{-}}\left( \func{sinc}(t-x)\right) ^{2}dtdx
\\
&=&\int_{0}^{\frac{u}{\pi }}\int_{\mathbb{R}}\left( \func{sinc}(t-x)\right)
^{2}dtdx-\int_{0}^{\frac{u}{\pi }}\int_{0}^{\frac{u}{\pi }}\left( \func{sinc}%
(t-x)\right) ^{2}dtdx \\
&=&\frac{u}{\pi }-\int_{0}^{\frac{u}{\pi }}\int_{0}^{\frac{u}{\pi }}\left(
\func{sinc}(t-x)\right) ^{2}dtdx\text{,}
\end{eqnarray*}%
the last equality being true because $\int_{\mathbb{R}}\func{sinc}%
(t-x)^{2}dt=1$. Now, Landau's inequality \cite[(8)]{HJLandau_1967ieee} gives
\begin{equation*}
\int_{0}^{\frac{u}{\pi }}\int_{0}^{\frac{u}{\pi }}\left( \func{sinc}%
(t-x)\right) ^{2}dtdx\geq \frac{u}{\pi }-C\log \left( \frac{u}{\pi }\right)
-B.
\end{equation*}%
It follows that
\begin{equation*}
\mathcal{S}(y,u)\leq C\log \left( \frac{u}{\pi }\right) -B.
\end{equation*}%
Thus, for some positive constants $K$ and $L$ not depending on $r$,
\begin{equation*}
I(y,u)\leq K\log u+L.
\end{equation*}
\end{proof}

\subsubsection{Proof of Proposition \protect\ref{norm_complete_estimation}.}

\begin{proof}
Since the integrand is always non-negative, the double sum in (\ref{Norm_sum}%
) is bounded below by any of the single sums. Thus,
\begin{equation}
\func{Norm}\geq
\sum_{i=1}^{n}\int_{b_{i}}^{b_{i}+s_{i}}\int_{b_{i}}^{b_{i}+s_{i}}\left[
(a+r)\mathcal{R}_{\alpha }((a+r)t,(a+r)x)-a\mathcal{R}_{\alpha }(at,ax)%
\right] ^{2}dtdx.  \label{norm_est1}
\end{equation}

Denote each of the double integrals in the sum (\ref{norm_est1}) by $\func{N}%
_{i}$. The duality (\ref{inverterpapeis}) between $I$ and $S$ gives
\begin{equation}
\func{N}_{i}=\int_{a}^{a+r}\int_{a}^{a+r}\left[ (b_{i}+s_{i})\mathcal{R}%
_{\alpha }((b_{i}+s_{i})t,(b_{i}+s_{i})x)-b_{i}\mathcal{R}_{\alpha
}(b_{i}t,b_{i}x)\right] ^{2}dtdx.
\end{equation}%
Set
\begin{equation}
\mathcal{M}_{\alpha }(x,t;b_{i},s_{i})=(b_{i}+s_{i})\mathcal{R}_{\alpha
}((b_{i}+s_{i})t,(b_{i}+s_{i})x)-b_{i}\mathcal{R}_{\alpha }(b_{i}t,b_{i}x).
\label{Malpha}
\end{equation}%
Using Lemma \ref{lemma_normestimation_T} with $b=b_{i}$ and $s=s_{i}$, we
get that there exists a contant $L_{0}$, not depending on $a$ or $r$, such
that

\begin{equation}
\func{N}_{i}\geq \frac{1}{\pi }rs_{i}-L_{0}-\int_{0}^{a}\int_{a}^{a+r}%
\mathcal{M}_{\alpha
}^{2}(x,t;b_{i},s_{i})dtdx-\int_{a}^{a+r}\int_{a+r}^{\infty }\mathcal{M}%
_{\alpha }^{2}(x,t;b_{i},s_{i}).  \label{depoisdotruquedotraco_M}
\end{equation}%
Applying the inequality $(a+b)^{2}\leq 2(a^{2}+b^{2})$ to (\ref{Malpha})
gives
\begin{equation*}
\mathcal{M}_{\alpha }^{2}(x,t;b_{i},s_{i})\leq 2(b_{i}+s_{i})^{2}\mathcal{R}%
_{\alpha }^{2}((b_{i}+s_{i})x,(b_{i}+s_{i})t)+2b_{i}^{2}\mathcal{R}_{\alpha
}^{2}(b_{i}x,b_{i}t).
\end{equation*}%
Plugging in (\ref{depoisdotruquedotraco_M}), and doing a change of variable,
we get

\begin{eqnarray}
\func{N}_{i} &\geq &\frac{1}{\pi }%
rs_{i}-L_{0}-2P((b_{i}+s_{i})a,(b_{i}+s_{i})r)-2P(b_{i}a,b_{i}r)  \notag
\label{depoisdotruquedotraco_R} \\
&&-2Q((b_{i}+s_{i})a,(b_{i}+s_{i})r)-2Q(b_{i}a,b_{i}r),
\end{eqnarray}%
with,
\begin{equation}
P(a,r)=\int_{0}^{a}\int_{a}^{a+r}\mathcal{R}_{\alpha }^{2}(x,t)dtdx,
\label{depoisdotruquedotraco_P}
\end{equation}%
and%
\begin{equation}
Q(a,r)=\int_{a}^{a+r}\int_{a+r}^{\infty }\mathcal{R}_{\alpha }^{2}(x,t)dtdx.
\label{depoisdotruquedotraco_Q}
\end{equation}%
Therefore, Lemma \ref{lemma_normestimation_PQ}, provides constants $%
K_{1},K_{2}$ and $L_{1}$ not depending on $a$ or $r$ such that%
\begin{equation}
\func{N}_{i}\geq \frac{1}{\pi }rs_{i}-L_{0}-K_{1}\log
((b_{i}+s_{i})r)-K_{1}\log (b_{i}r)-K_{2}\log ((b_{i}+s_{i})r)-K_{2}\log
(b_{i}r)-L_{1}.
\end{equation}%
Finally, from the inequality above we find constans $K$ and $L$ not
depending on $a$ or $r$ such that

\begin{equation}
\func{N}_{i}\geq \frac{1}{\pi }rs_{i}-K\log (r)-L.
\end{equation}%
Plugging this on (\ref{norm_est1}) we get (for new constants $K$ and $L$),

\begin{equation}
\func{Norm}\geq \frac{1}{\pi }rs-K\log (r)-L.
\end{equation}
\end{proof}

\section{\protect\bigskip Sampling and interpolation and the eigenvalue
problem}

In this section we will prove a series of results required to connect the
sampling and interpolation problem to the eigenvalue problem of the previous
section. Here we will make use of the convolution structure associated to
the Hankel transform.

Write $A\lesssim B$ to signify that $A\leq CB$ for some constant $C>0$,
independent of whatever arguments are involved. The next proposition is the
analogue of Proposition 1 in \cite{HJLandau_1967ieee}.

\begin{proposition}
Let $S$ be bounded, and let $\{t_{n}\}$ be a set of interpolation for $%
B_{\alpha }(S)$. Then the points of $\{t_{n}\}$ are separated by at least
some positive distance $d$, and the interpolation can be performed in a
stable way.
\end{proposition}

\begin{proof}
From%
\begin{equation*}
f(t)=\int_{S}H_{\alpha }f(x)J_{\alpha }(tx)(tx)^{\frac{1}{2}}dx
\end{equation*}
one has
\begin{equation*}
|f(t)|^{2}\lesssim \int_{0}^{\infty }|H_\alpha f(x)|^{2}dx =
\int_{0}^{\infty }|f(x)|^{2}dx.
\end{equation*}%
Likewise, the identity
\begin{equation*}
f^{\prime }(t)=\int_{S}H_{\alpha }f(x)\frac{\partial J_{\alpha }(tx)(tx)^{%
\frac{1}{2}}}{\partial t}dx
\end{equation*}%
provides a similar estimate for $|f^{\prime }(t)|$. 
The rest of the proof completely follows Landau~\cite{HJLandau_1967ieee}.
\end{proof}

Now we will prove the lemmas corresponding to Lemma 1 and Lemma 2 in \cite%
{HJLandau_1967ieee}.

\begin{lemma}
\label{techlemmasampling} Let $S$ be bounded and $\{t_{n}\}$ a set of
sampling for $B_{\alpha }(S)$, whose points are separated by at least $2d>0$
. Let $I$ be any compact set, $I^{+}$ be the set of points whose distance to
$I$ is less than $d$, and $n(I^{+})$ be the number of points of $\{t_{n}\}$
contained in $I^{+}$. Then $\lambda _{n(I^{+})}(I,S)\leq \gamma <1$, where $%
\gamma $ depends on $S$ and $\{t_{n}\}$ but not in $I$.
\end{lemma}

\begin{proof}
To adapt the arguments of \cite{HJLandau_1967ieee}, we need the convolution
structure associated with the Hankel transform outlined in section 2. Let $h$
be a function with support in $[0,d]$ such that we can find constants $%
K_{1},K_{2}$ such that its Hankel transform satisfies
\begin{equation}
K_{1}x^{\alpha +\frac{1}{2}}\leq (H_{\alpha }h)(x)\leq K_{2}x^{\alpha +\frac{%
1}{2}}\text{,}  \label{esth}
\end{equation}%
for every $x\in S$. In order to construct such a function, we use the fact
(see \cite[pag. 482]{GNWatson_1944}) that, if $A$ and $B$ are real (not both
zero) and $\alpha >-1$, then the function $AJ_{\alpha }(z)+BzJ_{\alpha
}^{\prime }(z)$ has all its zeros on the real axis, except that it has two
purely imaginary ones when $A/B+\alpha <0$. Thus, if $z_{0}$ is a complex
number outside the imaginary and the real axis, then the function%
\begin{equation*}
x^{-\alpha -\frac{1}{2}}\mathcal{R}_{\alpha }(x,z_{0})=z_{0}^{\frac{1}{2}}%
\frac{J_{\alpha }(z_{0})x^{-(\alpha -1)}J_{\alpha }^{\prime }(x)-x^{-\alpha
}J_{\alpha }(x)z_{0}J_{\alpha }^{\prime }(z_{0})}{z_{0}^{2}-x^{2}}
\end{equation*}%
is bounded away from zero and infinity for every $x\in S$. Since
\begin{equation*}
\mathcal{R}(x,t)=H_{\alpha }(\chi _{\lbrack 0,1]}J_{\alpha }(\cdot t)(\cdot
t)^{1/2})(x),
\end{equation*}%
then the function $h$ defined via its Hankel transform as
\begin{equation*}
(H_{\alpha }h)(x)=\mathcal{R}_{\alpha }(dx,z_{0})
\end{equation*}%
has the desired property. Now define
\begin{equation*}
g(x):=f\ast _{\alpha }h(x)=x^{\alpha +\frac{1}{2}}\left( \frac{\pi }{2}%
\right) ^{\frac{\alpha }{2}}\frac{1}{\Gamma (\alpha +1)}\int_{0}^{\infty
}f(t)\tau _{x}h(t)dt.
\end{equation*}%
Since
\begin{equation*}
H_{\alpha }g(x)=H_{\alpha }(f\ast _{\alpha }h)(x)=x^{-\alpha -\frac{1}{2}%
}(2\pi )^{\frac{\alpha }{2}}H_{\alpha }f(x)H_{\alpha }h(x)
\end{equation*}%
and $f\in B_{\alpha }(S)$ then clearly also $g\in B_{\alpha }(S).$ It
follows that
\begin{equation*}
\Vert g\Vert ^{2}\lesssim \sum_{n}|g(t_{n})|^{2}.
\end{equation*}%
Since $K_{1}x^{\alpha +\frac{1}{2}}\leq H_{\alpha }h(x)$, then%
\begin{equation*}
\Vert f\Vert ^{2}\lesssim \Vert g\Vert ^{2},
\end{equation*}

By formula (\ref{Proposition1}), $\func{supp}\tau _{x}h\subset \lbrack \max
\{0,x-d\},x+d]$ and one can write
\begin{eqnarray*}
|g(x)|^{2} &\leq &\left( x^{\alpha +\frac{1}{2}}\left( \frac{\pi }{2}\right)
^{\frac{\alpha }{2}}\frac{1}{\Gamma (\alpha +1)}\int_{0}^{\infty }f(t)\tau
_{x}h(t)dt\right) ^{2} \\
&\leq &x^{2\alpha +1}\left( \left( \frac{\pi }{2}\right) ^{\frac{\alpha }{2}}%
\frac{1}{\Gamma (\alpha +1)}\right) ^{2}\left( \int_{x-d}^{x+d}f(t)\tau
_{x}h(t)dt\right) ^{2} \\
&\lesssim &x^{2\alpha +1}\left( \int_{x-d}^{x+d}f(t)\tau _{x}h(t)dt\right)
^{2} \\
&\lesssim &\left( \int_{x-d}^{x+d}\tau _{x}h(t)^{2}x^{2\alpha +1}dt\right)
\int_{x-d}^{x+d}|f(t)|^{2}dt.
\end{eqnarray*}%
Moreover,{\normalsize
\begin{eqnarray*}
\int_{x-d}^{x+d}\tau _{x}h(t)^{2}x^{2\alpha +1}dt &=&x^{2\alpha +1}\Vert
\tau _{x}h\Vert ^{2} \\
&=&x^{2\alpha +1}\Vert j_{\alpha }(x\cdot )H_{\alpha }h(\cdot )\Vert ^{2} \\
&=&\int_{S}x^{2\alpha +1}j_{\alpha }^{2}(xs)H_{\alpha }h(s)^{2}ds \\
&=&\int_{S}\left( (xs)^{\alpha +\frac{1}{2}}j_{\alpha }(xs)\right)
^{2}\left( \frac{H_{\alpha }h(s)}{s^{\alpha +\frac{1}{2}}}\right) ^{2}ds \\
&=&C_{\alpha }^{\prime }\int_{S}xsJ_{\alpha }(xs)^{2}\left( \frac{H_{\alpha
}h(s)}{s^{\alpha +\frac{1}{2}}}\right) ^{2}ds \\
&\lesssim &\func{m}(S)\int_{S}xsJ_{\alpha }(xs)^{2}ds \\
&\lesssim &\int_{0}^{r_{0}}xsJ_{\alpha }(xs)^{2}ds\text{, for some }r_{0}. \\
&\leq &C\text{,}
\end{eqnarray*}%
for some constant $C>0$. }We have thus shown that
\begin{equation}
|g(x)|^{2}\lesssim \int_{x-d}^{x+d}|f(t)|^{2}dt\text{.}  \label{inequalityg}
\end{equation}%
Now we impose the $n(I^{+})$ orthogonality conditions:
\begin{equation*}
\int_{0}^{\infty }f(t)\tau _{t_{n}}h(t)dt=0
\end{equation*}%
for every $t_{n}\in I^{+}$. This gives $g(t_{n})=0$, for every $t_{n}\in
I^{+}$.Finally, using the separation of $\{t_{n}\}$ and the definition of $%
I^{+}$,
\begin{equation*}
\Vert f\Vert ^{2}\leq \Vert g\Vert ^{2}\leq K\sum_{t_{n}\notin
I^{+}}|g(t_{n})|^{2}\leq K^{\prime }\sum_{t_{n}\notin
I^{+}}\int_{t_{n}-d}^{t_{n}+d}|f(t)|^{2}dt\leq K^{\prime }\int_{\mathbb{R}%
^{+}\setminus I}|f(t)|^{2}dt
\end{equation*}%
\begin{equation*}
\frac{1}{\Vert f\Vert ^{2}}\int_{I}|f(t)|^{2}dt=1-\frac{1}{\Vert f\Vert
_{2}^{2}}\int_{\mathbb{R}^{+}\setminus I}|f(t)|^{2}dt\leq 1-\frac{1}{%
K^{\prime }}<1.
\end{equation*}%
Since $K^{\prime }$ is independent of $I$, the Lemma is proved.
\end{proof}

\begin{lemma}
\label{techlemmainterpolation}Let $S$ be bounded and $\{t_{n}\}$ a set of
interpolation for $B_{\alpha }(S)$, whose points are separated by at least $%
d>0$. Let $I$ be any compact set, $I^{-}$ be the set of points whose
distance to the complement of $I$ exceeds $\frac{d}{2}$, and $n(I^{-})$ be
the number of points of $\{t_{n}\}$ contained in $I^{-}$. Then $\lambda
_{n(I^{-})-1}(I,S)\geq \delta >0$, where $\delta $ depends on $S$ and $%
\{t_{n}\}$ but not in $I$.
\end{lemma}

\begin{proof}
We have shown in Proposition 1 that the interpolation can be done in a
stable way, thus
\begin{equation*}
\Vert g\Vert ^{2}\leq K\sum_{n}|g(t_{n})|^{2}.
\end{equation*}%
Now, for each $t_{l}$ let $\phi _{l}\in B(S)$ be the interpolating function
that is $1$ at $t_{l}$ and $0$ in the rest of the $t_{n}$, all these
functions are linearly independent. Let $h$ be the same as in the proof of
Lemma \ref{techlemmasampling} and define $\psi _{l}\in B_{\alpha }(S)$ by
\begin{equation*}
(\psi _{l})(x)=H_{\alpha }\left( \frac{(.)^{\alpha +\frac{1}{2}}}{\left(
2\pi \right) ^{\frac{\alpha }{2}}}\frac{H_{\alpha }\phi _{l}(.)}{H_{\alpha
}h(.)}\right) \text{,}
\end{equation*}%
for every $x\in S$ (recall that $x^{-\alpha -\frac{1}{2}}H_{\alpha }h$ is
bounded away from zero in $S$). By Hankel transform,
\begin{equation*}
\phi _{n}=\psi _{n}\ast _{\alpha }h\text{.}
\end{equation*}%
Given $f\in span\{\psi _{n}\,\}_{t_{n}\in I^{-}}$, let $g=f\ast _{\alpha }h$%
. Then $g$ is a linear combination of $\phi _{n}$ with $t_{n}\in I^{-}$,
thus $g(t_{n})=0$ for $t_{n}\notin I^{-}$. We get, using the estimates of
the proof of Lemma~ \ref{techlemmasampling} leading to (\ref{inequalityg}),
\begin{equation*}
\Vert f\Vert ^{2}\leq K\Vert g\Vert ^{2}\leq K^{\prime }\sum_{t_{n}\in
I^{-}}|g(t_{n})|^{2}\leq K^{\prime \prime }\sum_{t_{n}\in
I^{-}}\int_{t_{n}-d}^{t_{n}+d}|f(t)|^{2}dt\leq K^{\prime \prime
}\int_{I}|f(t)|^{2}dt
\end{equation*}%
so
\begin{equation*}
\lambda _{k-1}(I,S)\geq \inf_{f\in \func{span}\{\psi _{n}\,\}_{t_{n}\in
I^{-}}}\frac{{\int_{I}|f(t)|^{2}dt}}{{\int_{0}^{\infty }|f(t)|^{2}dt}}\geq
\frac{1}{K^{\prime \prime }}\text{.}
\end{equation*}%
Once again $K^{\prime \prime }$ does not depend on $I$ and we are done.
\end{proof}

The information contained in the proof of the above Lemmas, namely
inequality (\ref{inequalityg}), allows to show that the separation implies
the upper inequality in the definitions of sampling and the frame properties.

\begin{proposition}
Let $S$ be bounded, and let $\{t_{n}\}$ be a set of points separated by at
least some positive distance $d$. Then, there exists a constant $B$ such
that, for every $f\in B_{\alpha }(S)$, we have%
\begin{equation*}
\sum_{n}\left\vert g(t_{n})\right\vert ^{2}\leq B\Vert g\Vert ^{2}.
\end{equation*}
\end{proposition}

\begin{proof}
Let $g\in B_{\alpha }(S)$ and $h$ the function constructed in Lemma~ \ref%
{techlemmasampling}. Write
\begin{equation*}
f=H_{\alpha }\left( \frac{(.)^{\alpha +\frac{1}{2}}}{\left( 2\pi \right) ^{%
\frac{\alpha }{2}}}\frac{H_{\alpha }g(.)}{H_{\alpha }h(.)}\right) .
\end{equation*}%
Clearly $f\in B_{\alpha }(S)$ and by construction $g=f\ast _{\alpha }h$.
Thus the estimate (\ref{esth}) gives%
\begin{equation*}
\Vert f\Vert ^{2}=\Vert H_{\alpha }f\Vert ^{2}\leq C\Vert (.)^{-\alpha -%
\frac{1}{2}}(H_{\alpha }h)(.)(H_{\alpha }f)(.)\Vert ^{2}=C\Vert H_{\alpha
}g\Vert ^{2}=C\Vert g\Vert ^{2},
\end{equation*}%
with $C$ independent of $g$. In the proof of Lemma 4 we have seen that if $%
g=f\ast _{\alpha }h$ then (\ref{inequalityg}) holds, allowing us to write%
\begin{equation*}
\sum_{n}\left\vert g(t_{n})\right\vert ^{2}\leq
C\sum_{n}\int_{x-d}^{x+d}|f(t)|^{2}dt=C\Vert f\Vert ^{2}\leq C^{\prime
}\Vert g\Vert ^{2}\text{.}
\end{equation*}
\end{proof}

\section{Proof of the main results}

Let $\{t_{n}\}$ be a sequence with separation constant $d$. We will denote
by $S^{+}$ the set of points whose distance to $S$ is less than $d$ and by $%
S^{-}$ the set of points whose distance to the complement of $S$ exceeds $d$%
. We will use the notation $\lfloor x\rfloor $ to denote the largest integer
smaller or equal than $x$.

Using the identities (\ref{inverterpapeis}) and (\ref{rescalar}) one can see
that $\lambda _{k-1}(r,S)$ is also the $k^{th}$ eigenvalue of the problem of
concentrating on the set $rS$ the functions whose Hankel transform is
supported in $[0,1]$. The sampling theorem associated with the Hankel
transform \cite{JRHiggins_1972} (together with the \cite[(2.3)]%
{JJGuadalupe_MPerez_FJRuiz_JLVarona_1993}) states that $\{j_{\alpha ,n}\}$
is a sequence of both sampling and interpolation, and is known to be a
perturbation of the set $\{\frac{n}{\pi }\}$. Then, there exists $\Upsilon =%
\mathcal{O}(1)$ such that, when $S$ is the union of $N$ intervals, the
number of these points contained in $S^{+}$ is at most $\lfloor \frac{1}{\pi
}r\func{m}(S)\rfloor +\Upsilon N$ and their number in $S^{-}$ is at least $%
\lfloor \frac{1}{\pi }r\func{m}(S)\rfloor -\Upsilon N$. Now, from Lemma 5
and Lemma 6 of Section 5, there are $\gamma _{0}$, $\delta _{0}$ such that
\begin{eqnarray}
\lambda _{\lfloor \frac{1}{\pi }r\func{m}(S)\rfloor +\Upsilon N}(r,S) &\leq
&\gamma _{0}<1  \label{classico+} \\
\lambda _{\lfloor \frac{1}{\pi }r\func{m}(S)\rfloor -\Upsilon N-1}(r,S)
&\geq &\delta _{0}>0.  \label{classico-}
\end{eqnarray}

\begin{theorem}[Sampling]
Let $S$ be a finite union of intervals. If $\{t_{n}\}$ is a set of sampling
for $B_{\alpha }(S)$, then $[0,r]$ must contain at least $(\frac{1}{\pi }r%
\func{m}(S)-A\log r-B)$ points of $\{t_{n}\}$, with $A$ and $B$ constants
not depending on $r$.
\end{theorem}

\begin{proof}
Let $\{t_{n}\}$ is a set of sampling for $B_{\alpha }(S)$. By Lemma~\ref%
{techlemmasampling} there exists $\gamma $ independent of $r$ such that
\begin{equation*}
\lambda _{n(I)+2}\leq \lambda _{n(I^{+})}\leq \gamma <1.
\end{equation*}%
From (\ref{classico-}),
\begin{equation*}
\lambda _{\left\lfloor \frac{1}{\pi }r\func{m}(S)\right\rfloor -\Upsilon
N-1}(r,S)\geq \delta _{0}>0.
\end{equation*}%
As the number of eigenvalues between $\delta _{0}$ and $\gamma $ increase at
most logarithmically with $r$, we have
\begin{equation*}
\lfloor \frac{1}{\pi }r\func{m}(S)\rfloor -\Upsilon N-1-n(I)+2\leq A^{\prime
}\log r+B^{\prime }.
\end{equation*}%
Thus,
\begin{equation*}
n(I)\geq \frac{1}{\pi }r\func{m}(S)-A\log r-B,
\end{equation*}%
for some $A$ and $B$ not depending on $r$. $\hfill \Box $
\end{proof}

\begin{theorem}[Interpolation]
Let $S$ be a finite union of intervals. If $\{t_{n}\}$ is a set of
interpolation for $B_{\alpha }(S)$, then $[0,r]$ must not contain more than $%
(\frac{1}{\pi }r\func{m}(S)-C\log r-D)$ points of $\{t_{n}\}$, with $C$ and $%
D$ constants not depending on $r$.
\end{theorem}

\begin{proof}
Let $\{t_{n}\}$ be a set of interpolation for $B_{\alpha }(S)$. By Lemma \ref%
{techlemmainterpolation} there exists $\delta $ independent of $r$ such that
\begin{equation*}
\lambda _{n(I)-3}\geq \lambda _{n(I^{-})}\geq \delta >0.
\end{equation*}%
From (\ref{classico+}),
\begin{equation*}
\lambda _{\lfloor \frac{1}{\pi }r\func{m}(S)\rfloor +\Upsilon N}(r,S)\leq
\gamma _{0}<1.
\end{equation*}%
As the number of eigenvalues between $\delta $ and $\gamma _{0}$ increase at
most logarithmically with $r$ we have
\begin{equation*}
(n(I)-3)-\left( \lfloor \frac{1}{\pi }r\func{m}(S)\rfloor +\Upsilon N\right)
\leq C^{\prime }\log r+D^{\prime }
\end{equation*}%
Thus,
\begin{equation*}
n(I)\leq \frac{1}{\pi }r\func{m}(S)+C\log r+D,
\end{equation*}%
for constants $C,D$ not depending on $r$. $\hfill \Box $
\end{proof}

To extend the result to more general sets, we can proceed as in Landau \cite[%
pag. 49-50]{HJLandau_1967acta}. In the sampling case, the result can be
extended to a general measurable set $S$ by observing that it suffices to
prove the result for compact sets and then cover a compact $S$ set by a
finite collection of intervals with disjoint interiors and measure arbitrary
close to the measure of $S$. In the interpolation case, the result is
extended to bounded measurable sets by approximating in measure from the
outside by bounded open sets.

Finally we remark that the case of functions whose Hankel transform is
supported on $\left[ a,a+r\right] $ cannot be reduced to a case where a
sequence of both sampling and interpolation is known to exists (since,
unlike the Fourier case, our eigenvalue problem is not translation
invariant). Nevertheless, we can still obtain asymptotic versions of the
inequalities which can be used to prove Theorem 1 and Theorem 2. This
problem is also present in \cite{JMarzo2007}, who offers a solution which
can be adapted to our setting. From Lemma 4 we know that
\begin{equation*}
\#\{\lambda _{j}(I,S)>\gamma \}\leq n(I^{+})\leq n(I)+o(r),\text{ \ \ }%
r\rightarrow \infty .
\end{equation*}%
and using exactly the same argument of \cite[page 582]{JMarzo2007}, we can
obtain the lower estimate%
\begin{equation*}
\#\{\lambda _{j}(I,S)>\gamma \}\geq \func{Trace}-\frac{1}{1-\gamma }\left(
\func{Trace}-\func{Norm}\right) .
\end{equation*}%
Thus,%
\begin{equation*}
n(I)\geq \frac{1}{\pi }rm(S)-\frac{1}{1-\gamma }A\log r-B-o(r)\text{ \ \ }%
r\rightarrow \infty .
\end{equation*}%
The estimate required for interpolation can be obtained in a similar way.
Now, Theorem 1 and Theorem 2 are straightforward consequences of the
definitions of lower and upper density.

\subsection*{Acknowledgments}

We would like to thank Richard Laugesen, \'{O}scar Ciaurri, Juan Luis Varona
for discussions related to the problem treated in this work and to thank
Kristian Seip for pointing us reference \cite{JMarzo2007}. We would also
like to thank the referee for suggesting several improvements and to the
referees of an earlier version of the paper, who hold the paper to a higher
standart by challenging us to prove the stronger \textquotedblleft
translation-invariant\textquotedblright\ \ results contained in this version.

\bibliographystyle{plain}
\bibliography{refs-LH}

\end{document}